\documentclass[a4paper,12pt,twoside]{amsart}

\usepackage[cm]{fullpage}

\usepackage{mymacros,oitp}

\title[Modular forms of $\mathrm{O}(2,4;\mathbb{Z})$]{The ring of modular forms of $\mathrm{O}(2,4;\mathbb{Z})$ with characters}
\author[A.~Nagano]{Atsuhira Nagano}
\address{
Faculty of Mathematics and Physics,
Institute of Science and Engineering,
Kanazawa University,
Kakuma, Kanazawa, Ishikawa,
920-1192, Japan}
\email{atsuhira.nagano@gmail.com}
\author[K.~Ueda]{Kazushi Ueda}
\address{
Graduate School of Mathematical Sciences,
The University of Tokyo,
3-8-1 Komaba,
Meguro-ku,
Tokyo,
153-8914,
Japan.}
\email{kazushi@ms.u-tokyo.ac.jp}
\date{}

\begin{document}

\begin{abstract}
We show that the ring of modular forms with characters
for $\mathrm{O}(2,4;\mathbb{Z})$
is generated by forms of weights 4, 4, 6, 8, 10, 10, 12, and 30
with three relations of weights 8, 20, and 60.
The proof is based on the study of a moduli space of K3 surfaces.
\end{abstract}

\maketitle

\section{Introduction}

Given a primitive sublattice $P$
of signature $(1,\rho-1)$
of the K3 lattice
$
 \KL \coloneqq E_8 \bot E_8 \bot U \bot U \bot U,
$
let
$
 T = T_P \coloneqq P^\bot
$
be the orthogonal lattice and
$\Gamma = \Gamma_P$ be the subgroup of $\rO(\KL)$
consisting of elements $g$
satisfying $g|_P = \id_P$.
By abuse of notation,
we write the image of $\Gamma$
under the injection
$
 g \mapsto g|_{T} \in \rO(T)
$
by the same symbol.
A \emph{$P$-polarized K3 surface}
in the sense of Nikulin \cite{MR544937}
is a pair
$(Y, j)$ of a K3 surface $Y$
and a primitive lattice embedding
$
 j \colon P \hookrightarrow \Pic Y.
$
As explained in \cite[Section 3]{MR1420220},
the global Torelli theorem \cite{MR0284440, MR0447635}
and the surjectivity of the period map \cite{MR592693}
shows that the period map gives an isomorphism
from the coarse moduli scheme of
pseudo-ample $P$-polarized K3 surfaces
to the quotient
$
 M \coloneqq \cD/\Gamma
$
of the bounded Hermitian domain
\begin{align} \label{eq:D}
 \cD \coloneqq \lc \ld \Omega \rd \in \bfP(T \otimes \bC)
  \relmid (\Omega, \Omega) = 0, \ (\Omega, \Omegabar) > 0 \rc
\end{align}
of type IV.

Let
$T$ be a lattice of signature $(2,n)$ for a positive integer $n$
and
\begin{align}
 \cDtilde \coloneqq \lc \Omega \in T \otimes \bC \relmid (\Omega, \Omega) = 0, \ (\Omega, \Omegabar) > 0 \rc
\end{align}
be the total space of a principal $\bCx$-bundle over $\cD$
defined by \pref{eq:D}.
For a discrete subgroup $\Gamma$
of finite covolume in $\rO(T \otimes \bR)$,
a \emph{modular form}
on $\cD$
with respect to $\Gamma$
of weight $k \in \bZ$
and character $\chi \in \Char(\Gamma) \coloneqq \Hom(\Gamma, \bCx)$
is a holomorphic function
$
 f \colon \cDtilde \to \bC
$
satisfying
\begin{enumerate}[(i)]
 \item
$f(\alpha z) = \alpha^{-k} f(z)$
for any $\alpha \in \bCx$, and
 \item
$f(\gamma z) = \chi(\gamma) f(z)$
for any $\gamma \in \Gamma$.
\end{enumerate}
The vector spaces
$A_k(\Gamma, \chi)$
of modular forms
constitute the ring
\begin{align}
 \Atilde(\Gamma)
  \coloneqq \bigoplus_{k=0}^\infty \bigoplus_{\chi \in \Char(\Gamma)}
      A_k(\Gamma, \chi)
\end{align}
of modular forms.
We also write the subring of modular forms without characters as
\begin{align}
 A(\Gamma)
  \coloneqq \bigoplus_{k=0}^\infty
      A_k(\Gamma).
\end{align}
The main result of this paper is the following:

\begin{theorem} \label{th:main}
The graded ring of modular forms with characters
of
$\rO(2,4;\bZ)$
is generated by modular forms of weights
4, 4, 6, 8, 10, 10, 12, and 30
with three relations of weights 8, 20, and 60;
\begin{align}
  \Atilde(\rO(2,4;\bZ)) \cong \bC[t_4,t_6,t_8,t_{10},t_{12},s_4,s_{10},s_{30}]
   / (s_4^2 -\Delta_8(t), s_{10}^2-\Delta_{20}(t),s_{30}^2-\Delta_{60}(t)),
\end{align}
where the polynomials $\Delta_8(t)$, $\Delta_{20}(t)$, and $\Delta_{60}(t)$
are given in \pref{eq:Delta_8}, \pref{eq:Delta_20}, and \pref{eq:Delta_60}
respectively.
\end{theorem}

This paper is organized as follows:
In \pref{sc:UE7E7},
we prove that the coarse moduli scheme of
$U \bot E_7 \bot E_7$-polarized K3 surfaces
is the double cover of the weighted projective space
$\bfP(4,6,8,10,12)$
branched along the divisor defined by $\Delta_{20}(t)$.
We give a proof of \pref{th:main} in \pref{sc:O24}.
In \pref{sc:six_lines},
we discuss the relation with the configuration space of six lines on $\bfP^2$
following \cite{MR1136204}.

\emph{Acknowledgment}:
We thank Kenji Hashimoto for valuable discussions,
and the anonymous referee for suggestions for improvement.
A.~N.~was partially supported by JSPS Kakenhi (18K13383, MEXT LEADER).
K.~U.~was partially supported by JSPS Kakenhi (15KT0105, 16K13743, 16H03930).

\section{The coarse moduli space
of $U \bot E_7 \bot E_7$-polarized K3 surfaces}
 \label{sc:UE7E7}

Consider the even lattice
$
 P_1 \coloneqq U \bot E_7 \bot E_7
$
of signature $(1,15)$, and
fix a primitive embedding of $P_1$ into the K3 lattice $\KL$,
which is unique up to the action of $\rO(\KL)$
by \cite[Theorem 1.14.4]{MR525944}.
The orthogonal lattice $T_1 \coloneqq P_1^\bot$
is isometric to $U \bot U \bot A_1 \bot A_1$.
We write $\Gamma_1 \coloneqq \Gamma_{P_1}$.
Let $\sigma_1 \in \rO(T_1)$ be the reflection
along the $(-4)$-vector
obtained as the sum of the bases $(-2)$-vectors
of two $A_1$ summands in $T_1$.
It fixes $U \bot U \subset T_1$
and interchanges two $A_1$ summands.
It extends to an element of $\rO(\KL)$
whose restriction to $P_1$ interchanges two $E_7$ summands.
One can easily see that
the group
$
 \rO(T_1) / \Gamma_1
$
is generated by the class of $\sigma_1$
(see the proof of \pref{pr:Gamma3} below).

\begin{theorem} \label{th:main2}
The graded ring $A(\Gamma_1)$ of automorphic forms
with respect to $\Gamma_1$
is given by
\begin{align}
 A(\Gamma_1)
  = \bC[t_4, t_6, t_8, t_{10}, s_{10}, t_{12}] / (s_{10}^2 - \Delta_{20}(t)),
\end{align}
where the lower indices indicate the weights,
and $\Delta_{20}(t)$ is defined in \pref{eq:Delta_20}.
\end{theorem}

\begin{proof}
Giving a $P_1$-polarized K3 surface is equivalent to giving an elliptic K3 surface
with a section and two singular fibers containing an $E_7$-configuration
(i.e., of Kodaira type $\mathrm{II}^*$ or $\mathrm{III}^*$),
which we may assume to lie above $0$ and $\infty$
on the base $\bfP^1$.
An elliptic K3 surface with a section
admits a Weierstrass model of the form
\begin{align} \label{eq:weierstrass1}
 z^2 = y^3 + g_2(x,w) y + g_3(x,w)
\end{align}
in $\bfP(1,4,6,1)$
(cf.~e.g.~\cite[Section 4]{MR2732092}).
Recall that the elliptic surface \eqref{eq:weierstrass1}
has a singular fiber of type either $\mathrm{II}^*$ or $\mathrm{III}^*$
at $a \in \bfP^1$
only if
$
 \ord_a g_2(x,w) \ge 3
$
and
$
 \ord_a g_3(x,w) \ge 5
$
(cf.~e.g.~\cite[Table IV.3.1]{MR1078016}).
This requires
\begin{align}
 g_2(x,w) &= u_{5,3} x^5 w^3 + u_{4,4} x^4 w^4 + u_{3,5} x^3 w^5,
  \label{eq:weierstrass2} \\
 g_3(x,w) &= u_{7,5} x^7 w^5 + u_{6,6} x^6 w^6 + u_{5,7} x^5 w^7.
  \label{eq:weierstrass3}
\end{align}
It has a singularity worse than
rational double points
on the fiber at $a \in \bfP^1$
if and only if $\ord_a(g_2) \ge 4$ and
$\ord_a(g_3) \ge 6$
(cf.~e.g.~\cite[Proposition I\!I\!I.3.2]{MR1078016}).
This is the case if and only if either
$u_{3,5}=u_{5,7}=0$ (for $a = 0$)
or $u_{5,3}=u_{7,5}=0$
(for $a=\infty$).
The parameter
\begin{align}
 u = (u_{5,3}, u_{4,4}, u_{3,5}, u_{7,5}, u_{6,6}, u_{5,7})
  \in U \coloneqq \bC^6 \setminus
   \lc u_{3,5}=u_{5,7}=0 \text{ or } u_{5,3}=u_{7,5}=0 \rc
\end{align}
appearing in the Weierstrass model \pref{eq:weierstrass1}
is unique up to the action of $(\bCx)^2$
given by
\begin{align}
 \bCx \ni \lambda \colon
  ((x,y,z,w),(u_{i,j})_{i,j}) \mapsto (x,\lambda^{2} y,\lambda^{3} z,w), (\lambda^{(i+j)/2} u_{i,j})_{i,j})
\end{align}
and
\begin{align}
 \bCx \ni \mu \colon
  ((x,y,z,w),(u_{i,j})_{i,j}) \mapsto ((\mu^{-1} x,y,z, \mu w), (\mu^{i-j} u_{i,j})_{i,j}).
\end{align}
Note that the former $\bCx$-action rescales the holomorphic volume form
\begin{align}
 \Omega= \Res \frac{w dx \wedge dy \wedge dz}{z^2 - y^3 - g_2(x,w;u) y - g_3(x,w;u)}
\end{align}
as
\begin{align}
 \Omega_{\lambda u}
  &= \Res \frac{w dx \wedge d(\lambda^2 y) \wedge d (\lambda^3 z)}{(\lambda^3 z)^2 - (\lambda^2 y)^3 - g_2(x,w;\lambda \cdot u) (\lambda^2 y) - g_3(x,w;\lambda \cdot u)}
 = \lambda^{-1} \Omega_u,
\end{align}
whereas the latter
(which comes from the automorphism of the base $\bfP^1$
fixing $0$ and $\infty$)
keeps it invariant.
The categorical quotient
$
 T \coloneqq U / \bCx_\mu
$
is the coarse moduli scheme of pairs
$(Y, \Omega)$ consisting of a $P_1$-polarized K3 surface $Y$
and a holomorphic volume form $\Omega \in H^0(\omega_Y)$ on $Y$.
The coordinate ring $\bC[T]$ of $T$
is generated by six elements
\begin{align}
 t_4 &\coloneqq u_{4,4}, &
 t_6 &\coloneqq u_{6,6}, &
 t_8 &\coloneqq u_{5,3} u_{3,5}, \\
 t_{10} &\coloneqq u_{5,3} u_{5,7} + u_{3,5} u_{7,5}, &
 s_{10} &\coloneqq u_{5,3} u_{5,7} - u_{3,5} u_{7,5}, &
 t_{12} &\coloneqq u_{7,5} u_{5,7}
\end{align}
with one relation
\begin{align} \label{eq:Delta_20}
 s_{10}^2 = \Delta_{20}(t) \coloneqq t_{10}^2 - 4 t_8 t_{12}.
\end{align}
The boundary of the affinization
$
 \Tbar \coloneqq \Spec \bC[T]
$
is given by
\begin{align}
 \lc t_8 = t_{10} = s_{10} = t_{12} = 0 \rc \cong \bC_{t_4} \times \bC_{t_6}.
\end{align}
The period map induces an isomorphism
of the graded ring of modular forms
and the coordinate ring $\bC[T]$.
The weight of the modular form is identified
with the weight of the $\bCx_\lambda$-action, and
\pref{th:main} is proved.
\end{proof}

\section{Modular forms of $\rO(2,4;\bZ)$ with characters}
 \label{sc:O24}

The lattice $T_1$ has a unique extension
$
 T_1 \subset T_2 \subset T_1 \otimes \bQ
$
of index 2
to an odd unimodular lattice
$
 T_2
  \cong U^{\perp 2} \bot \la -1 \ra^{\bot 2}
  \cong \la 1 \ra^{\bot 2} \bot \la -1 \ra^{\bot 4},
$
and $T_1$ is the sublattice of $T_2$
consisting of even elements;
\begin{align}
 T_1 \cong \lc v \in T_2 \relmid \pair{v}{v}  \in 2 \bZ \rc.
\end{align}
It follows that
$
 \rO(T_2)
  \cong \rO(2,4;\bZ)
$
can naturally be identified with $\rO(T_1)$
as a subgroup of $\rO(T_1 \otimes \bQ)$,
so that
\begin{align}
 A(\rO(T_2)) \cong A(\rO(T_1)) = A(\Gamma_1)^{\la \sigma_1 \ra}.
\end{align}
Since $\sigma_1$ acts on $U$ by sending $u_{i,j}$ to $u_{j,i}$,
one obtains a proof of the following:

\begin{theorem}[{\cite[Theorem 1]{MR2718942}}]
 \label{th:Vinberg}
The graded ring $A(\rO(2,4;\bZ))$ of automorphic forms
is given by
\begin{align}
 A(\rO(2,4;\bZ))
  = \bC[t_4, t_6, t_8, t_{10}, t_{12}]
\end{align}
where the lower indices indicate the weights.
\end{theorem}

In fact,
this proof of \pref{th:Vinberg} already appears
in \cite{MR4015343}.
In particular,
\pref{eq:weierstrass1}
is identical to \cite[(4.13)]{MR4015343}
up to an obvious change of coordinates.
Note that `isomorphisms of $H \oplus E_7(-1) \oplus E_7(-1)$
lattice polarized K3 surfaces'
in \cite[Proposition 4.3.(b)]{MR4015343}
come from the action of $\sigma_1$,
and as such are not isomorphisms of lattice polarized K3 surfaces
in the sense of \cite{MR1420220}.

The coarse moduli space
$
 M \coloneqq \cD/\rO(T_1)
$
of $P_1$-polarized K3 surfaces
up to the action of $\sigma_1$
is an open subvariety
of its Satake--Baily--Borel compactification
$
 \Proj A(\rO(T_1))
  \cong \bfP(4,6,8,10,12).
$
Although $M$ and
the the orbifold quotient
$
 \bM \coloneqq \ld \cD/\rO(T_1) \rd
$
are closely related,
the canonical morphism $\bM \to M$
is not an isomorphism
even in codimension 1.
In order to obtain an orbifold
which is isomorphic to $\bM$ in codimension 1
(so that
the total coordinate rings are isomorphic),
we first consider the stacky weighted projective space
$
 \bP \coloneqq \bP(4,6,8,10,12),
$
defined as the quotient of $\bC^5 \setminus \bszero$
by a $\bCx$-action with this weight.
The morphism
$
 \bM \to M
$
lifts to a morphism
$
 \bM \to \bP,
$
which is an isomorphism in codimension 0,
since the generic stabilizers are $\lc \pm \id \rc$
on both sides.

Stabilizers of $\bM$ along divisors come from
reflections,
and besides $\sigma_1$ appearing above,
there are two more reflections that one can easily find
in $\rO(T_1)$.
The first one, which we call $\sigma_2$,
is the reflection along the $(-2)$-vector
whose reflection hyperplane is defined by
\begin{align} \label{eq:Delta_8}
 \Delta_8(t) \coloneqq t_8.
\end{align}
In terms of $P_1$-polarized K3 surfaces,
this divisor corresponds to the locus
where the Picard lattice contains
$
 U \bot E_7 \bot E_8.
$
The second one, which we call $\sigma_3$,
is the reflection along the $(-2)$-vector
whose reflection hyperplane
corresponds to the locus
where the Picard lattice contains
$U \bot E_7 \bot E_7 \bot A_1$.
In order to describe this locus,
first consider the discriminant
$
 4 g_2(x,w;t)^3 + 27 g_3(x,w;t)^2
$
of
$
 y^3 + g_2(x,w;t) y + g_3(x,w;t)
$
as a polynomial in $y$,
which is the product of $x^9 w^9$
and a homogeneous polynomial $h(x,w;t)$,
of degree 6 in $(x,w)$
and degree 12 in $t$.
Note that the discriminant of the polynomial
$
 \sum_{i=0}^n a_i x^i w^{n-i}
$
with respect to $(x,w)$
is homogeneous of degree $2(n-1)$
in $\bZ[a_0, \ldots, a_n]$
if $\deg a_0 = \cdots = \deg a_n = 1$.
It follows that the discriminant
$k_{120}(t)$ of $h(x,w;t)$ with respect to $(x,w)$
is a homogeneous polynomial
of degree $2 \cdot 5 \cdot 12 = 120$ in $t$.
A general point
on the divisor of $\bfP(4,6,8,10,12)$
defined by $k_{120}(t)$
corresponds to the locus
where two fibers of Kodaira type $\mathrm{I}_1$ collapse
into one fiber.
This divisor has two components;
a general point on one corresponds to the case
when there exists a point $p = [x:w] $ on $\bfP^1$ such that
neither $g_2$ nor $g_3$ vanishes at $p$,
and a general point on the other component
corresponds to the case
when both $g_2$ and $g_3$ vanishes at $p$.
In the former case,
the resulting singular fiber is of Kodaira type $\mathrm{I}_2$,
and the surface acquires an $A_1$-singularity.
In the latter case,
the resulting singular fiber is of Kodaira type I\!I,
and the surface does not acquire any new singularity.
The defining equation of the latter component
is the resultant of $g_2$ and $g_3$.
It is given as the determinant
\begin{align}
 r_{20}(t)=
\left|
\begin{array}{cccc}
 u_{5,3} & u_{4,4} & u_{3,5} & \\
 & u_{5,3} & u_{4,4} & u_{3,5} \\
 u_{7,5} & u_{6,6} & u_{5,7} & \\
 & u_{7,5} & u_{6,6} & u_{5,7}
\end{array}
\right|
\end{align}
of the Sylvester matrix,
which is homogeneous of degree 20.
As shown in \cite[Lemma 6.1]{1406.0332},
the polynomial $k_{120}(t)$ is divisible by $r_{20}(t)^3$, and
a direct calculation using a computer algebra system
shows that the quotient
\begin{align} \label{eq:Delta_60}
 \Delta_{60}(t) \coloneqq k_{120}(t)/r_{20}(t)^3
\end{align}
is irreducible.

Recall from
\cite{MR2450211,MR2306040}
that the \emph{root construction}
is an operation
which adds a stabilizer
along a divisor.
Let $\bT$ be the stack
obtained from $\bP$
by the root construction of order 2
along the divisor on $\bP$
defined by $\Delta_{88}(t) \coloneqq \Delta_8(t) \Delta_{20}(t) \Delta_{60} (t)$,
which is the quotient of the double cover of $\bP$
branched along $\Delta_{88}(t)$
by the group $G$ of deck transformations.
The Picard group of $\bT$
(or the $G$-equivariant Picard group of $\bP$)
is generated by the pull-back
$
 \cO_{\bT}(1) \coloneqq p^* \cO_{\bP}(1)
$
of the generator $\cO_{\bP}(1)$ of the Picard group of $\bP$
by the structure morphism
$
 p \colon \bT \to \bP
$
and three line bundles $\cO_{\bT}(\bD_i)$ for $i=4,10,30$
such that the space
$
 H^0(\cO_{\bT}(\bD_i))
$
is generated by an element $s_i$
satisfying
$
 s_i^2
  = \Delta_i
  \in H^0(\cO_{\bT}(i))
  \cong H^0(\cO_{\bP}(i)).
$
The ramification formula for the canonical bundle gives
\begin{align}
 \omega_{\bT}
  &\cong p^* \omega_{\bP} \otimes \cO_{\bT}(\bD_4+\bD_{10}+\bD_{30}) \\
  &\cong \cO_{\bT}(-40) \otimes \cO_{\bT}(\bD_4+\bD_{10}+\bD_{30}) \\
  &\cong \cO_{\bT}(4) \otimes \cO_{\bT}(-44+\bD_4+\bD_{10}+\bD_{30}).
   \label{eq:omega_T}
\end{align}
Note that $\cO_{\bT}(-44+\bD_4+\bD_{10}+\bD_{30})$ is an element
of order two in $\Pic \bT$.
By comparing \pref{eq:omega_T} with
\begin{align} \label{eq:omega}
 \omega_{\bM} \cong \cO_{\bM}(4) \otimes \det
\end{align}
which follows from (the proof of) \cite[Proposition 5.1]{1406.0332},
one concludes that $\bM$ has no further stabilizer along a divisor,
so that the lift $\bM \to \bT$ of $\bM \to \bP$ is an isomorphism
in codimension 1.
It follows that $\Pic \bM \cong \Pic \bT$
is isomorphic to $\bZ \oplus (\bZ/2\bZ)^3$,
and the total coordinate ring
(also known as the Cox ring) of $\bM$ is given by
\begin{align} \label{eq:TCR}
 \bigoplus_{\cL \in \Pic \bM} H^0(\cL)
  \cong \bC[t_4,t_6,t_8,t_{10},t_{12},s_4,s_{10},s_{30}]
   / (s_4^2 -\Delta_8(t), s_{10}^2-\Delta_{20}(t),s_{30}^2-\Delta_{60}(t)).
\end{align}
A character of $\rO(T_1)$ gives a line bundle on $\bM$,
so that the ring $\Atilde(\rO(T_1))$ of modular forms with characters
is a subring of the total coordinate ring,
which in fact is the whole of it
since any line bundle on $\bM$ comes from a character
in the case at hand.
This can be seen by noting that
the universal cover $\cD \to \bM$ factors
through a $(\bZ/2\bZ)^3$-cover
defined by the equation
appearing on the right hand side of \pref{eq:TCR},
so that the line bundles $\cO_\bT(-i+\bD_i)$
for $i=4,10,30$ come from characters
of the orbifold fundamental group of $\bM$,
which is isomorphic to $\rO(T_1)$ by definition.

Note that
\pref{eq:cong} below produces
two distinct elements of order two
in $\Char(\rO(T_1))$;
one comes from the sign representation of $\fS_6$
and the other comes from the unique non-trivial representation of
$\la \tau \ra$.
Yet another element of order two comes from
the determinant representation
$
 \det \colon \rO(T_1) \to \{ \pm 1 \},
$
which does not belong to the subgroup of $\Char(\rO(T_1))$
generated by the above two characters.
Since $\det(g) = -1$ for any reflection $g \in \rO(T_1)$,
the character of the modular form $s_4 s_{10} s_{30}$ is $\det$.
As we explain in \pref{sc:six_lines},
the characters of the modular forms $s_{10}$ and $s_{30}$
come from the non-trivial representation of $\la \tau \ra$
and the sign representation of $\fS_6$ respectively.

\section{Configuration of six lines on the plane}
 \label{sc:six_lines}

Let $L_1, \ldots, L_6$ be six lines on $\bfP^2$
in very general position,
and $Y$ be the K3 surface
obtained as the resolution
of 15 ordinary double points
on the double cover of $\bfP^2$
branched along the union of these six lines.
Fix an isometry $H^2(Y;\bZ) \cong \KL$ and
regard the Picard lattice $P_3$ of $Y$
as a sublattice of $\KL$.
The lattice $P_3$ is an even lattice of signature $(1,15)$
generated by the classes of strict transforms of the lines
and 15 exceptional divisors.
The primitive embedding of $P_3$ into $\KL$ is unique up to the action of $\rO(\KL)$
by \cite[Theorem 1.14.4]{MR525944}.
The orthogonal lattice $T_3$ of $P_3$
inside $\KL$ is isometric to
$
 T_2(2) \cong U(2) \bot U(2) \bot A_1 \bot A_1,
$
so that one has
\begin{align}
 \rO(T_3) \cong \rO(T_2).
\end{align}
Set
$
 \Gamma_3 \coloneqq \Gamma_{P_3}
$
and
$
 \Gamma(2) \coloneqq \lc g \in \rO(T_3) \relmid g \equiv \id_{T_3} \mod 2 \rc.
$
The intersection
$
\Gamma(2)^+ \coloneqq \Gamma(2) \cap \rO^+(T_3)
$
of $\Gamma(2)$
with the subgroup $\rO^+(T_3) \subset \rO(T_3)$ of index two
preserving the connected component $\cD^+$ of $\cD = \cD^+ \cup \cD^-$
is the monodromy group of the hypergeometric function of type (3,6)
studied extensively in \cite{MR1136204}.

\begin{proposition} \label{pr:Gamma3}
One has
$
 \Gamma_3 = \Gamma(2).
$
\end{proposition}

\begin{proof}
It follows from \cite[Corollary 1.5.2]{MR525944} that
$\Gamma_3$ is the kernel of the natural homomorphism
from $\rO(T_3)$
to the group of automorphisms
of the discriminant group of $T_3$
(i.e., the quotient of the dual lattice
$T_3^\dual \coloneqq \Hom(T_3, \bZ)$
by the natural injection $T_3 \hookrightarrow T_3^\dual$).

It is shown in \cite[Proposition 2.7.3 and Corollary 2.7.4]{MR1136204}
that
$
\Gamma(2)^+
$
is a reflection group
generated by reflections along 20 $(-2)$-vectors
given on \cite[p.~103]{MR1136204}.
One can easily check that every reflection acts
on $T_3^\vee/T_3$ as the identity,
so that $\Gamma(2) \subset \Gamma_3$.

On the other hand,
it is shown in \cite[Proposition 2.8.2]{MR1136204} that
the quotient group
$
 \rO(T_3)/\Gamma(2)
$
is the finite group given by
\begin{align}
 \lc
  \begin{pmatrix} Y_4 &0 \\ 0 & \eta_2 \end{pmatrix}
 \relmid
  Y_4 \in \Sp(2, \bZ/2\bZ), \,
  \eta_2 \in \{I_2,U\}
 \rc,
\end{align}
where
\begin{align}
 \Sp(2,\bZ/2\bZ) \coloneqq
  \lc Y_4 \in \GL(4, \bZ/2\bZ) \relmid
   Y_4^T (U \oplus U) Y_4 = U\oplus U \rc.
\end{align}
Note that
$
 \Sp(2,\bZ/2\bZ)
$
is generated by
\begin{align}
 \begin{pmatrix} 1 & 1 & \\ 0 & 1& \\& & I_2 \end{pmatrix},
 \begin{pmatrix} U & \\ & I_2 \end{pmatrix},
 \begin{pmatrix} 1 &  & \\ & U& \\& & 1 \end{pmatrix},
 \begin{pmatrix} I_2 & \\ & U \end{pmatrix},
 \ \text{and} \ 
 \begin{pmatrix} I_2 &  & \\  & 1& 1 \\ &0 & 1 \end{pmatrix},
\end{align}
and isomorphic to the symmetric group $\fS_6$ of degree 6.
One can easily see that
if $h \in \rO(T_3)/\Gamma(2)$ is not the identity,
then $h$ induces a non-trivial transformation on $T_3^\vee/T_3$,
so that $\Gamma(2) = \Gamma_3$.
\end{proof}

We write the element of order 2
in $\rO(T_3) / \Gamma_3$
represented by $I_4 \oplus U$ as $\tau$,
so that
\begin{align} \label{eq:cong}
 \rO(T_3)/\Gamma_3 \cong \fS_6 \times \la \tau \ra.
\end{align}

The \emph{Igusa quartic} is
the Siegel modular variety of genus 2,
which can naturally be identified with the moduli spaces of
\begin{itemize}
 \item
principally polarized abelian surfaces,
 \item
genus two curves, and
 \item
hyperelliptic curves of genus two,
i.e., configurations of six points on $\bfP^1$.
\end{itemize}
It can be described as a quartic hypersurface in $\bfP^4$,
defined by the equations
\begin{align}
 \sum_{i=0}^5 x_i &= 0, \\
 \lb \sum_{i=0}^5 x_i^2 \rb^2 &= 4 \sum_{i=0}^5 x_i^4
\end{align}
in $\bfP^5$.
The GIT quotient
$
 \Xbar(3,6)
$
of
$
 (\bfP^2)^6
$
by the action of $\PGL_3 \cong \Aut \bfP^2$
with respect to the democratic weight
is known to be the double cover of $\bfP^4$
branched along the Igusa quartic
by \cite{MR1007155}.
The period map
from the configuration space $X(3,6)$
of six lines on $\bfP^2$
in general position
to the modular variety
$
 M_3 \coloneqq \cD / \Gamma_3
$
extends to an isomorphism
from $\Xbar(3,6)$
to the Satake--Baily--Borel compactification $\Mbar_3$
of $M_3$
by \cite{MR1136204}.
As explained in \cite[Section A.2]{MR1136204},
the action of $\tau$ on $M_3$
gives an involution on $M_3$
whose fixed locus is the moduli space of six points on a conic,
which can naturally be identified
with the Igusa quartic;
the natural projection
$
 \Mbar_3 \to \Mbar_3 /\la \tau \ra \cong \bfP^4
$
is a double cover of $\bfP^4$
branched along the Igusa quartic.
The residual action of $\fS_6$ on $M_3 / \la \tau \ra$
is the projectivization
of the the natural action of $\fS_6$ on
$
 \lc (x_1, \ldots, x_6) \in \bA^6 \relmid x_1 + \cdots + x_6 = 0 \rc
$
by permutation of coordinates.
The quotient
\begin{align}
 \cD/\rO(T_3)
  \cong M_3 / \fS_6 \times \la \tau \ra
  \cong \bfP^4 / \fS_6
  \cong \Spec A(\rO(T_3))
\end{align}
is the weighted projective space
$
 \bfP(2,3,4,5,6) = \Proj \bC[t_4,t_6,t_8,t_{10},t_{12}],
$
where $t_{2i}$ are symmetric functions of degree $i$ in $x_1, \ldots, x_6$.
The projection
\begin{align}
 M_1 \coloneqq \cD/\Gamma_1
  \to \cD/\rO(T_1) \cong \bfP(2,3,4,5,6)
\end{align}
is the double cover branched along
the hypersurface defined by
$\Delta_{20}(t).$
It follows that the character of $s_{10}$
is the composite of
$
 \rO(2,4;\bZ)
  \cong \rO(T_3)
  \to \rO(T_3)/\Gamma_3
  \cong \fS_6 \times \la \tau \ra
  \to \la \tau \ra
$ 
and the non-trivial representation of $\la \tau \ra$.

Since the branch locus of the double cover of $\bfP(2,3,4,5,6)$
associated with the sign representation of $\fS_6$
is the discriminant $\Delta_{60}(t) = \prod_{1 \le i < j \le 6}(x_i-x_j)^2$,
the character of the modular form $s_{30}$ is the composite
of the surjection $\rO(2,4;\bZ) \to \fS_6$
and the sign representation of $\fS_6$.

\bibliographystyle{amsalpha}
\bibliography{bibs}

\def\cprime{$'$} \def\cprime{$'$}
\providecommand{\bysame}{\leavevmode\hbox to3em{\hrulefill}\thinspace}
\providecommand{\MR}{\relax\ifhmode\unskip\space\fi MR }
\providecommand{\MRhref}[2]{%
  \href{http://www.ams.org/mathscinet-getitem?mr=#1}{#2}
}
\providecommand{\href}[2]{#2}
\begin{thebibliography}{AGV08}

\bibitem[AGV08]{MR2450211}
Dan Abramovich, Tom Graber, and Angelo Vistoli, \emph{Gromov-{W}itten theory of
  {D}eligne-{M}umford stacks}, Amer. J. Math. \textbf{130} (2008), no.~5,
  1337--1398. \MR{2450211 (2009k:14108)}

\bibitem[BR75]{MR0447635}
Dan Burns, Jr. and Michael Rapoport, \emph{On the {T}orelli problem for
  k\"{a}hlerian {$K-3$} surfaces}, Ann. Sci. \'{E}cole Norm. Sup. (4)
  \textbf{8} (1975), no.~2, 235--273. \MR{0447635}

\bibitem[Cad07]{MR2306040}
Charles Cadman, \emph{Using stacks to impose tangency conditions on curves},
  Amer. J. Math. \textbf{129} (2007), no.~2, 405--427. \MR{2306040
  (2008g:14016)}

\bibitem[CMS19]{MR4015343}
A.~Clingher, A.~Malmendier, and T.~Shaska, \emph{Six line configurations and
  string dualities}, Comm. Math. Phys. \textbf{371} (2019), no.~1, 159--196.
  \MR{4015343}

\bibitem[DO88]{MR1007155}
Igor Dolgachev and David Ortland, \emph{Point sets in projective spaces and
  theta functions}, Ast\'{e}risque (1988), no.~165, 210 pp. (1989).
  \MR{1007155}

\bibitem[Dol96]{MR1420220}
I.~V. Dolgachev, \emph{Mirror symmetry for lattice polarized {$K3$} surfaces},
  J. Math. Sci. \textbf{81} (1996), no.~3, 2599--2630, Algebraic geometry, 4.
  \MR{1420220 (97i:14024)}

\bibitem[HU]{1406.0332}
Kenji Hashimoto and Kazushi Ueda, \emph{The ring of modular forms for the even
  unimodular lattice of signature (2,10)}, arXiv:1406.0332.

\bibitem[Mir89]{MR1078016}
Rick Miranda, \emph{The basic theory of elliptic surfaces}, Dottorato di
  Ricerca in Matematica. [Doctorate in Mathematical Research], ETS Editrice,
  Pisa, 1989. \MR{1078016 (92e:14032)}

\bibitem[MSY92]{MR1136204}
Keiji Matsumoto, Takeshi Sasaki, and Masaaki Yoshida, \emph{The monodromy of
  the period map of a {$4$}-parameter family of {$K3$} surfaces and the
  hypergeometric function of type {$(3,6)$}}, Internat. J. Math. \textbf{3}
  (1992), no.~1, 164. \MR{1136204}

\bibitem[Nik79a]{MR544937}
V.~V. Nikulin, \emph{Finite groups of automorphisms of {K}\"ahlerian {$K3$}
  surfaces}, Trudy Moskov. Mat. Obshch. \textbf{38} (1979), 75--137. \MR{544937
  (81e:32033)}

\bibitem[Nik79b]{MR525944}
\bysame, \emph{Integer symmetric bilinear forms and some of their geometric
  applications}, Izv. Akad. Nauk SSSR Ser. Mat. \textbf{43} (1979), no.~1,
  111--177, 238. \MR{525944 (80j:10031)}

\bibitem[P{\v{S}}{\v{S}}71]{MR0284440}
I.~I. Pjatecki{\u\i}-{\v{S}}apiro and I.~R. {\v{S}}afarevi{\v{c}},
  \emph{Torelli's theorem for algebraic surfaces of type {${\rm K}3$}}, Izv.
  Akad. Nauk SSSR Ser. Mat. \textbf{35} (1971), 530--572. \MR{0284440 (44
  \#1666)}

\bibitem[SS10]{MR2732092}
Matthias Sch{\"u}tt and Tetsuji Shioda, \emph{Elliptic surfaces}, Algebraic
  geometry in {E}ast {A}sia---{S}eoul 2008, Adv. Stud. Pure Math., vol.~60,
  Math. Soc. Japan, Tokyo, 2010, pp.~51--160. \MR{2732092 (2012b:14069)}

\bibitem[Tod80]{MR592693}
Andrei~N. Todorov, \emph{Applications of the
  {K}\"{a}hler-{E}instein-{C}alabi-{Y}au metric to moduli of {$K3$} surfaces},
  Invent. Math. \textbf{61} (1980), no.~3, 251--265. \MR{592693}

\bibitem[Vin10]{MR2718942}
E.~B. Vinberg, \emph{Some free algebras of automorphic forms on symmetric
  domains of type {IV}}, Transform. Groups \textbf{15} (2010), no.~3, 701--741.
  \MR{2718942}

\end{thebibliography}

\end{document}